\documentclass[11pt]{amsart}
\usepackage{fullpage,amsmath,amsthm,amsfonts,amssymb,graphicx,amscd,float,hyperref,enumitem,setspace}

\title{$Out(F_3)$ index realization}

\newtheorem{thm}{Theorem}[section]

\newtheorem{qst}[thm]{Question}

\newtheorem{rk}[thm]{Remark}
\newtheorem{ex}[thm]{Example}
\theoremstyle{definition}

\newtheorem*{oldlem}{Lemma}

\begin{document}
\title{$Out(F_3)$ Index Realization}
\author{Catherine Pfaff}

\begin{abstract}
By proving precisely which singularity index lists arise from the pair of invariant foliations for a pseudo-Anosov surface homeomorphism, Masur and Smillie \cite{ms93} determined a Teichm\"{u}ller flow invariant stratification of the space of quadratic differentials. In this paper we determine an analog to the theorem for $Out(F_3)$. That is, we determine which index lists permitted by the \cite{gjll} index sum inequality are achieved by fully irreducible outer automorphisms of the rank-$3$ free group.
\end{abstract}

\maketitle

\section{Introduction}

We let $Out(F_r)$ denote the outer automorphism group of the rank-$r$ free group. In this paper we prove realization results for an outer automorphism invariant dependent only on the conjugacy class (within $Out(F_r)$) of the outer automorphism, namely the ``index list.'' This work is motivated both by the important role index lists have played in mapping class group theory and by the role they are already playing in studying the dynamics of the groups $Out(F_r)$.

The outer automorphism groups have been studied for many years. More recent developments have encouraged and enabled rapid analysis of deep relationships between the mapping class groups and the $Out(F_r)$. For a compact surface $\Sigma$, the \emph{mapping class group $\mathcal{MCG}(\Sigma)$} is the group of isotopy classes of orientation-preserving homeomorphisms of $\Sigma$. The relationship between the mapping class groups and $Out(F_r)$ is particularly visible in rank $2$, where there are even isomorphisms $Out(F_2) \cong Out(\pi_1(\Sigma_{1,1})) \cong \mathcal{MCG}(\Sigma_{1,1})$ for the one-holed torus $\Sigma_{1,1}$. It can be noted that even in higher ranks, many outer automorphisms are still induced by homeomorphisms of compact surfaces with boundary. For future reference, such outer automorphisms are called \emph{geometric}.

While not necessary for following the content of this paper, we first briefly explain indices in the mapping class group setting to orient the reader more familiar with surface theory. The index list is an important invariant of a ``pseudo-Anosov'' mapping class. Pseudo-Anosovs are the most common mapping class group elements (see for example \cite{m11}) and are characterized by having a representative leaving invariant a pair of transverse measured singular minimal foliations. In \cite{ms93} Masur and Smillie determined precisely which singularity index lists, permitted by the Poincar\'e-Hopf index formula, come from these invariant foliations of pseudo-Anosovs. The stratification they give of the space of quadratic differentials is not only invariant under the Teichm\"{u}ller flow, but has been extensively studied in papers such as \cite{kz03}, \cite{l04}, \cite{l05}, \cite{emr12}, and \cite{z10}.

The index list for a pseudo-Anosov can identically be viewed in terms of its invariant foliation or in terms of its dual $R$-tree. In fact, the singularities of the invariant foliation, lifted to the universal cover, are in one-to-one correspondence with the branchpoints of the dual $R$-tree. In the respective settings, the index list has an entry of $1-\frac{k}{2}$ obtained by counting the number $k$ of prongs at the singularity or the valence of the branch point. Alternatively, one can ascertain the index list from singularities of the expanding invariant lamination (obtained as the limit of any simple closed curve under repeated application of the pseudo-Anosov) or from the invariant train track. Much of this theory can be found in \cite{flp79}.

A ``fully irreducible'' (iwip) outer automorphism is the most commonly used analogue to a pseudo-Anosov mapping class. An element $\phi \in Out(F_r)$ is \emph{fully irreducible} if no positive power $\phi^k$ fixes the conjugacy class of a proper free factor of $F_r$. The index theory for automorphisms of free groups dates back to the work of Cooper \cite{c87}, Dyer and Scott \cite{ds75}, Gersten \cite{g87}, and later Bestvina and Handel \cite{bh92} in understanding the fixed point sets for an automorphism. This theory fully transformed into an index theory in \cite{gjll}. 

As with a pseudo-Anosov acting on Teichm\"{u}ller space, a fully irreducible acts with north-south dynamics \cite{ll03} on the natural compactification of Culler-Vogtmann Outer space \cite{cv86}. Both the attracting and repelling points for the action are $R$-trees, denoted respectively $T^+_{\phi}$ and $T^-_{\phi}$. The repelling tree is an extension, to nongeometric fully irreducibles, of the dual tree to the invariant foliation for a pseudo-Anosov. As with a pseudo-Anosov, the index list for a fully irreducible, as defined in \cite{hm11}, has an entry of $1-\frac{k}{2}$ obtained by counting the valence $k$ of the branch point. The index list can again also be computed from the expanding lamination of \cite{bfh97}. For a fully irreducible the lamination can be obtained by applying an automorphism in the class repeatedly to any generator, then taking the closure. The description of the index list we use here (explained in Section \ref{s:Dfs}) uses the ``train track representative'' proved to exist for a fully irreducible in \cite{bh92}.

While the $Out(F_r)$ groups resemble mapping class groups, there is added depth to the $Out(F_r)$. A particularly good example of this arises when trying to generalize the Masur-Smillie pseudo-Anosov index theorem to nongeometric fully irreducibles. One facet of this depth is expanded upon in \cite{p12b}, \cite{p12c}, and \cite{p12d}, where we show that, unlike with pseudo-Anosovs, where the ideal Whitehead graph can be determined by the singularity index list, the ideal Whitehead graph actually gives a finer invariant of a fully irreducible giving, in particular, more detailed behavior of the lamination at a singularity. In this paper we focus on the fact that, instead of being restricted by an index sum equality, such as the Poincar\'e-Hopf index equality, the index sum for a fully irreducible is only restricted by an inequality. Gaboriau, J\"aeger, Levitt, and Lustig proved in \cite{gjll} that each fully irreducible $\phi \in Out(F_r)$ satisfies that index sum inequality $0 > i(\phi) \geq 1-r$. Here we revise their index definition to be invariant under taking powers and to have the sign consistent with the mapping class group case. If one takes an adequately high power (the ``rotationless power'' of \cite{fh11}), the definitions differ only by a sign change.

Index lists of geometric fully irreducibles are understood by the Masur-Smillie index theorem. Complexity of the nongeometric case prompted Handel and Mosher to ask (\cite{hm11} Question 6):

\begin{qst}{\label{Q:GJLL}} Which index types, satisfying $0 > i(\phi) \geq 1-r$, are achieved by nongeometric, fully irreducible $\phi \in Out(F_r)$?  \end{qst}

We answered the rank $3$ case with:

\begin{thm}{\label{T:IndexTheorem}} Each of the six possible index lists,  $(-\frac{1}{2}, -\frac{1}{2}, -\frac{1}{2})$, $(-\frac{1}{2}, -1)$, $(-\frac{3}{2})$, $(-\frac{1}{2}, -\frac{1}{2})$, $(-1)$, and $(-\frac{1}{2})$, whose index sums satisfy $0 > i(\phi) > 1-r$ are realized by a fully irreducible $\phi \in Out(F_3)$. \end{thm} 

One may notice that we restrict to looking at outer automorphisms for which the right-hand inequality is strict. This is because we focus on ageometric outer automorphisms, as defined in \cite{gjll}. While ageometrics are believed generic, there does exists a second class of nongeometric outer automorphisms, namely the parageometrics (which could be classified as nongeometric fully irreducible outer automorphisms with geometric attracting tree). These have been studied in papers such as \cite{hm07} and \cite{g05}, where in fact they show that the inverse of a parageometric is ageometric. It can additionally be noted that Bestvina and Feighn give in \cite{bf94} a nice description of the distinction between ageometrics, parageometrics, and geometrics.

It is proved in \cite{gl95} that a fully irreducible has geometric attracting tree precisely if the index sum satisfies $i(\phi) = 1-r$. Thus, like geometrics, parageometrics have index sum $i(\phi) = 1-r$. It would be interesting to understand whether index lists satisfying $i(\phi) = 1-r$, but not realized by geometrics, are in fact realized by parageometric outer automorphisms.

While this paper is the first to focus on index list realization, the index theory for free group outer automorphisms has in some directions already been extensively developed. In fact, there are three types of $Out(F_r)$ index invariants in the literature, those of \cite{gl95}, \cite{gjll}, and \cite{ch10}. The index of $\phi$, as defined and studied in \cite{gjll}, is equal to the geometric index of $T^+_{\phi}$, as established by Gaboriau-Levitt \cite{gl95} for more general $R$-trees. \cite{ch12} provides a relationship between the index of \cite{ch10} and the geometric index, as well as uses the index to relate different properties of the attracting and repelling tree for a fully irreducible. There are also even index realization results of a different nature. For example, \cite{jl09} gives examples of automorphisms with the maximal number of fixed points on $\partial F_r$, as dictated by a related inequality in \cite{gjll}. Focusing on an $Out(F_r)$-version of the Masur-Smillie theorem, we restrict attention to fully irreducibles and the \cite{gjll} index inequality given above.

\subsection*{Acknowledgements}

\indent The author would like to thank her thesis advisor Lee Mosher for his continued discussions and support. She would like to thank Thierry Coulbois for his invaluable computer program and assistance. She would like to thank Ilya Kapovich for helpful discussions and advice. And, finally, she would like to thank Martin Lustig for his interest in her work.

The author was supported by the ARCHIMEDE Labex (ANR-11-LABX- 0033) and the A*MIDEX project (ANR-11-IDEX-0001-02) funded by the ``Investissements d’Avenir'' French government program managed by the ANR.

\section{Definitions and Background}{\label{s:Dfs}}

\vskip10pt

\noindent \textbf{Train track representatives.}

\vskip1pt

Let $R_r$ denote the $r$-petaled rose (graph with one vertex and $r$ edges) together with an identification $\pi(R_r) \cong F_r$. A connected $1$-dimensional CW-complex $\Gamma$ such that each vertex has valence greater than two, together with a homotopy equivalence (\emph{marking}) $R_r \to \Gamma$, is called a \emph{marked graph}. Each $\phi \in Out(F_r)$ is represented by a homotopy equivalence $g \colon \Gamma \to \Gamma$ of a marked graph, where $\phi=g_*$. When each vertex of $\Gamma$ has valence greater than two and $g$ additionally sends vertices to vertices and satisfies that $g^k$ is locally injective on edge interiors for each $k>0$, one says $g$ is a \emph{train track (tt) representative} for $\phi$. In \cite{bh92}, Bestvina and Handel prove that a fully irreducible outer automorphism always has a train track representative. Many of the definitions and notation for discussing train track representatives were established in \cite{bh92} and \cite{bfh00}. We remind the reader here of a few that are relevant.

Let $g \colon \Gamma \to \Gamma$ be a train track representative of some $\phi \in Out(F_r)$. For each $x\in \Gamma$, we let $\mathcal{D}(x)$ denote the set of \emph{directions} at $x$, i.e. germs of initial segments of edges emanating from $x$. For an edge $e \in \mathcal{E}(\Gamma)$, we let $D_0(e)$ denote the initial direction of $e$. For a path $\gamma=e_1 \dots e_k$, define $D_0 \gamma = D_0(e_1)$. We denote the map of directions induced by $g$ by \emph{$Dg$}, i.e. $Dg(d)=D_0(g(e))$ for $d=D_0(e)$. A direction $d$ is \emph{periodic} if $Dg^k(d)=d$ for some $k>0$.

We call an unordered pair of directions $\{d_i, d_j\}$ a \emph{turn}. It is an \emph{illegal turn} for $g$ if $Dg^k(d_i) = Dg^k(d_j)$ for some $k$ and \emph{legal} otherwise. Considering the directions of an illegal turn equivalent, one can define an equivalence relation on the set of directions at a vertex. Each equivalence class is called a \emph{gate}. For a path $\gamma=e_1e_2 \dots e_{k-1}e_k$ in $\Gamma$, we say $\gamma$ \emph{takes} $\{\overline{e_i}, e_{i+1}\}$ for each $1 \leq i < k$.

A train track representative is \emph{reducible} if it has an invariant subgraph with a noncontractible component and is otherwise called \emph{irreducible}. An outer automorphism $\phi$ is \emph{fully irreducible} if every representative of every power is irreducible.

\vskip10pt

\noindent \textbf{Local Whitehead graphs, ideal Whitehead graphs, and index lists.}

\vskip1pt

We assume in this subsection that $g\colon \Gamma \to \Gamma$ is a train track representive of $\phi \in Out(F_r)$. We also assume that $g$ has no \emph{periodic Nielsen path (pNp)}, i.e. a nontrivial path $\rho$ in $\Gamma$ such that, for some $k$, $g^k(\rho) \simeq \rho$ rel endpoints. As it is used in Section \ref{S:ac}, we remark that $\rho$ is called an \emph{indivisible periodic Nielsen path (ipNp)} if it cannot be written as a nontrivial concatenation $\rho=\rho_1 \cdot \rho_2$, where $\rho_1$ and $\rho_2$ are nontrivial pNps.

The following definitions are from \cite{hm11}. One can reference \cite{p12a} for more extensive explanations of the definitions and their invariance. It is notable that, while we use a train track representative here to define the ideal Whitehead graph and index list for a fully irreducible outer automorphism, both the ideal Whitehead graph and index list are invariants of the outer automorphism. In fact, they are invariants of the conjugacy class within $Out(F_r)$ of the outer automorphism.

Let $g\colon \Gamma \to \Gamma$ a train track representive of $\phi \in Out(F_r)$ and $v$ a vertex of $\Gamma$. The \emph{local Whitehead graph} $\mathcal{LW}(g; v)$ for $g$ at $v$ has a vertex for each direction at $v$ and an edge connecting the vertices for $d_i$ and $d_j$ if there exists an edge $e$ of $\Gamma$ and $k>0$ so that $g^k(e)$ takes the turn $\{d_i, d_j \}$.

\noindent Restricting to periodic directions, one obtains a subgraph called the \emph{local stable Whitehead graph} $\mathcal{SW}(g; v)$. Still assuming $g$ has no pNp's, the \emph{ideal Whitehead graph $\mathcal{IW}(\phi)$ of $\phi$} is then isomorphic to the disjoint union $\bigsqcup \mathcal{SW}(g;v)$ taken over all vertices with at least three periodic directions.

Let $\phi$ be a nongeometric fully irreducible outer automorphism and let $C_1, \dots, C_l$ be the connected components of the ideal Whitehead graph $\mathcal{IW}(\phi)$. For each $j$, let $k_j$ denote the number of vertices of $C_j$. The index list for $\phi$ can be defined as
$$(i_1, \dots, i_j, \dots, i_l) = (1-\frac{k_1}{2}, \dots, 1-\frac{k_j}{2}, \dots, 1-\frac{k_l}{2}).$$
The index sum is then $i(\phi) = \sum i_j$.

\vskip10pt

\noindent \textbf{Full irreducibility criterion.}

\vskip1pt

The \emph{transition matrix} for a train track representative $g$ is the square matrix such that, for each $i$ and $j$, the $ij^{th}$ entry is the number of times $g(E_j)$ crosses $E_i$ in either direction. A transition matrix $A=[a_{ij}]$ is \emph{Perron-Frobenius (PF)} if there exists an $N$ such that, for all $k \geq N$, $A^k$ is strictly positive (see for example \cite{bh92}).

In order to show that our maps represent fully irreducible outer automorphisms, we use the ``Full Irreducibility Criterion (FIC)'' proved in \cite{p12c} (Lemma 4.1):

\begin{oldlem}(\emph{The Full Irreducibility Criterion (FIC)})
Let $g\colon \Gamma \to \Gamma$ be a pNp-free, irreducible train track representative of $\phi \in Out(F_r)$. Suppose that the transition matrix for $g$ is Perron-Frobenius and that all the local Whitehead graphs are connected. Then $\phi$ is fully irreducible.
\end{oldlem}

\section{Verification Algorithms}{\label{S:ac}}

In Theorem \ref{T:IndexTheorem}, we used a computer program \cite{c12} to verify that each example is indeed a train track representative of the correct rank and additionally has no pNp's.  We include here a procedure for finding by hand all pNp's of a train track map. This procedure is not too different from that given in \cite{p12c} and is that applied in \cite{hm11} Example 3.4.

We also include here procedures for computing by hand local Whitehead graphs and ideal Whitehead graphs. In the proof of Theorem \ref{T:IndexTheorem} we only use the procedures for pNp-free train track maps.

We leave verification of the validity of all of the procedures to the reader.

\vskip10pt

\noindent \textbf{Finding periodic Nielsen paths.}

\vskip1pt

Let $g\colon \Gamma \to \Gamma$ be a train track map and $\{T_1, \dots, T_n \}$ the set of illegal turns for $g$. The following procedure will identify if there exists an ipNp $\rho = \overline{\rho_1} \rho_2$ for $g$, where $\rho_1 = e_1\dots e_m$ and $\rho_2 = e_1'\dots e_{m'}'$ are edge paths (with possibly $e_m$ and $e_{m'}'$ being partial edges) and with illegal turn $T_i = \{D_0(e_1), D_0(e_1')\}$ $=\{d_1, d_1'\}$. Since pNps can be decomposed into ipNp, as such, one can find all pNps for $g$. We let $\rho_{1,k}=e_1\dots e_k$ and $\rho_{2,l}= e_1'\dots e_l'$ throughout the procedure. 

Suppose $Dg^{j-1}(e_1') \neq Dg^{j-1}(e_1)$ but $Dg^j(e_1') = Dg^j(e_1)$. One of the following holds:

\begin{itemize}
\item[\textbf{(A)}] Either $g^j(e_1)$ is the initial subpath of $g^j(e_1')$ or vice versa.
\item[\textbf{(B)}] $g^j(e_1)=\gamma \alpha_1$ and $g^j(e_1')=\gamma \alpha_2$ where $\{D_0(\alpha_1), D_0(\alpha_2)\}$ is a legal turn.
\item[\textbf{(C)}] $g^j(e_1)=\gamma \alpha_1$ and $g^j(e_1')=\gamma \alpha_2$ where $\{D_0(\alpha_1), D_0(\alpha_2)\}$ is an illegal turn. And either \newline
(i) $\{D_0(\alpha_1), D_0(\alpha_2)\} = T_i$ \newline
(ii) or $\{D_0(\alpha_1), D_0(\alpha_2)\} \neq T_i$.
\end{itemize}

In the case of (A), one must proceed to ``edge addition'' (see below) with $e_1 = \rho_{1,1}$ and $e_1' = \rho_{2,1}$. In the case of (B), there is no ipNp containing $T_i$. In the case of (C)(i), there exists an ipNp from a fixed point of $e_1$ to a fixed point in $e_1'$. In the case of (C)(ii), to reach an outcome, one must continue composing with $g$ until landing in the case of (A$'$), (B$'$), below or (C$'$)(i) or reaching a rotationless power:

\begin{itemize}
\item[\textbf{(A$'$)}] Either $g^j(\rho_{1,k})$ is the initial subpath of $g^j(\rho_{2,l})$ or vice versa.
\item[\textbf{(B$'$)}] $g^j(\rho_{1,k})=\gamma \alpha_1'$ and $g^j(\rho_{2,l})=\gamma \alpha_2'$ where $\{D_0(\alpha_1), D_0(\alpha_2)\}$ is a legal turn.
\item[\textbf{(C$'$)}] $g^j(\rho_{1,k})=\gamma \alpha_1'$ and $g^j(\rho_{2,l})=\gamma \alpha_2'$ where $\{D_0(\alpha_1), D_0(\alpha_2)\}$ is an illegal turn. And either \newline
(i) $\rho_{1,k} \subset \alpha_1'$ and $\rho_{2,l} \subset \alpha_2'$ \newline
(ii) or (i) does not hold.
\end{itemize}

In the case of (A$'$), proceed to ``edge addition.'' In the case of (B$'$), $\rho_{1,k}$ and $\rho_{2,l}$ could not yield an ipNp containing $T_i$. In the case of (C$'$)(i), there exists an ipNp from a fixed point of $e_k$ to a fixed point in $e_l'$. In the case of (C$'$)(ii), continue composing with $g$ until either one lands in the case of (A$'$), (B$'$), or (C$'$)(i) or one reaches a rotationless power, in which case (C$'$)(ii) would indicate there is no iNp.

\vspace{5pt}

\noindent \emph{Edge addition:}

Without generality loss (or by adjusting notation) we can assume $g^j(\rho_{1,k})$ is the initial subpath of $g^j(\rho_{2,l})$, so $g^j(\rho_{2,l}) = g^j(\rho_{1,k})\sigma$, for some legal path $\sigma$. Then $\rho_1$ would need to contain another edge $e_{k+1}$. With each choice of $e_{k+1}$ such that $\rho_{1,k}e_{k+1}$ is legal and $Dg^j(e_{k+1}) = \sigma$, one must continue to compose with $g$ until following the above procedure either leads to a pNp or shows the choice does not lead to a pNp.

\begin{rk} One can note that this procedure is finite, as there are only finitely many ipNp's and a bound on the length of an iNp (as described in \cite{bh92} Corollary 3.5 to be a consequence of the ``bounded cancellation lemma''). 
\end{rk}

\vskip10pt

\noindent \textbf{Computing index lists.}

\vskip1pt

Suppose $g\colon \Gamma \to \Gamma$ is a train track representative for an ageometric fully irreducible. And supposed $\Gamma$ has periodic vertices $v_1, \dots, v_k$. For each $1 \leq i \leq k$, let $n_i$ denote the number of gates at the vertex $v_i$. Define an equivalence relation on the set of all periodic points by: $x_i \sim x_j$ if there exists a pNp running from $x_i$ to $x_j$. Call an equivalence class a \emph{Nielsen class}. For a Nielsen class $N_i = \{x_1, \dots, x_n \}$, let $g_i$ denote the number of gates at $x_i$. Now let 
\begin{center}
$n_i = (\sum g_i) - \# \{\text{iNPs}$ $\rho$ $\text{such that both endpoints of}$ $\rho$ $\text{are in} N_i \}$.
\end{center}
The index list is then 
$$\{1-\frac{n_1}{2}, \dots, 1-\frac{n_t}{2} \},$$
where we only include nonzero entries. 

Notice that there are only finitely many nonzero entries, as there are only  finitely many iNPs and a periodic point $x_i$ that is not a vertex and not the endpoint of an iNP will have $1-\frac{x_i}{2}=0$. Additionally notice that one does not need to find all periodic points to make this computation, but only needs to consider Nielsen classes that contain a vertex or endpoint of a pNp.

\vskip10pt

\noindent \textbf{Computation of local Whitehead graphs.}

\vskip1pt

Now let $g\colon \Gamma \to \Gamma$ be any train track map. Recall that the local Whitehead graph $\mathcal{LW}(g; v)$ has a vertex for each direction at $v$ and an edge connecting the vertices for $d_i$ and $d_j$ if there is some edge $e$ of $\Gamma$ and some $k>0$ so that $g^k(e)$ traverses the turn $\{d_i, d_j \}$. We explain a finite procedure for computing all such $\{d_i, d_j \}$. We denote by $T$ this list of turns traversed by some $g^k(e)$. 

Enumerate the edges $e_i$ of $\Gamma$. For each $e_i$, find the list of turns traversed by $g(e_i)$.  Let 
$$\mathcal{T} = \{\{d_{i_1}, d_{j_1} \}, \dots, \{d_{i_m}, d_{j_m} \} \}$$ 
be the union of these lists. That is, each $\{d_{i_b}, d_{j_b} \} \in \mathcal{T}$ is a turn taken by some $g(e_i)$. We now construct the list of turns $T$ as follows: $T$ first off includes all elements of $\mathcal{T}$, but it will also include all $Dg^k(\{d_{i_b}, d_{j_b} \})$ where $k>0$. To ensure this algorithmically: 

Start with $\{d_{i_1}, d_{j_1} \}$. Add $Dg(\{d_{i_1}, d_{j_1} \})$, $Dg^2(\{d_{i_1}, d_{j_1} \})$, etc, to $\mathcal{D}$ until reaching some $Dg^N(\{d_{i_1}, d_{j_1} \})$ already in $\mathcal{D}$. Now do the same for $\{d_{i_2}, d_{j_2} \}$, for $\{d_{i_3}, d_{j_3} \}$, etc. 

Notice that, not only is this set $T$ finite, but it will contain fewer than $mR$ elements, where $R$ is the minimum rotationless power (see \cite{fh11}). So the procedure is finite.

\begin{ex}{\label{e:lwg}} We consider the train track map on the rose:
~\\
\vspace{-3mm}
$$
g =
\begin{cases} a \mapsto cab \\
b \mapsto ca \\
c \mapsto acab
\end{cases}
$$

Since the automorphism is positive, it is easily verified to be train track.

\noindent The direction map $Dg$ sends:

$a \mapsto c \mapsto a \mapsto \dots$

$b \mapsto c \mapsto \dots$

$c \mapsto a \mapsto \dots$

$\bar{a} \mapsto \bar{b} \mapsto \bar{a} \mapsto \dots$

$\bar{b} \mapsto \bar{a} \mapsto \dots$

$\bar{c} \mapsto \bar{b} \mapsto \dots$

\noindent The periodic directions are $a$,$c$,$\bar{a}$, and $\bar{b}$ and the gates are $\{a,b\}$, $\{c\}$, $\{\bar{a},\bar{c}\}$, and $\{\bar{b}\}$. Since there are four gates at the single vertex, the index list has a single entry $1-\frac{4}{2}=-1$.

The turns taken by $g(a)$ are $\{\bar{c},a\}$ and $\{\bar{a},b\}$. The turns taken by $g(b)$ are $\{\bar{c},a\}$.The turns taken by $g(c)$ are $\{\bar{a},c\}$, $\{\bar{c},a\}$, and $\{\bar{a},b\}$. Thus,
~\\
\vspace{-5mm}
$$\mathcal{T} = \{\{\bar{c},a\}, \{\bar{a},b\}, \{\bar{a},c\} \}.$$

\noindent $\{\bar{c},a\}$ maps to $\{\bar{b},c\}$ maps to $\{\bar{a},a\}$ maps to $\{\bar{b},c\}$, which is already in $T$. $\{\bar{a},b\}$ maps to $\{\bar{b},c\}$, which is already in $T$. $\{\bar{a},a\}$ maps to $\{\bar{b},a\}$ maps to $\{\bar{a},c\}$, which is already in $T$. So
~\\
\vspace{-5mm}
$$T = \{\{\bar{c},a\}, \{\bar{a},b\}, \{\bar{a},c\}, \{\bar{b},c\}, \{\bar{a},a\}, \{\bar{b},a\} \}.$$

The two illegal turns are $\{a,b\}$ and $\{\bar{a},\bar{c}\}$. We verify that there is no ipNp containing $\{a,b\}$ and leave the verification for $\{\bar{a},\bar{c}\}$ to the reader.

$a \mapsto cab$ and

$b \mapsto ca$.

\noindent Thus we are in the case of (A) with $e_1=b$, and there must be another edge $e_2$ after $b$. $e_2$ must satisfy that either $Dg(e_2)=a$ or $Dg(e_2)=b$. The only such possibility is $e_2=c$. So $\rho_{1,2}=bc$. We apply $g$ twice because after the first application, cancellation ends in the illegal turn $\{\bar{a},\bar{c}\}$, but not in the manner of (C')(i):

$a \mapsto cab \mapsto acabcabca$ and

$bc \mapsto caacab \mapsto acabcabcabacabcabca$.

\noindent We are now in the case of (A') and must add another edge $e_2'$ after $e_1'=a$. $e_2'$ must satisfy that either $Dg^2(e_2)=a$ or $Dg^2(e_2)=b$. So we must check $e_2'=a$ and $e_2'=b$. The cancellation of $g^3(aa)$ and $g^3(bc)$ ends with $\{\bar{b},\bar{c}\}$, which is a legal turn. And the same is true for the cancellation of $g^3(ab)$ and $g^3(bc)$. So we are done.

\end{ex}

\section{Main Theorem}

\begin{thm}{\label{T:IndexTheorem}} Each of the six possible index lists,  $(-\frac{1}{2}, -\frac{1}{2}, -\frac{1}{2})$, $(-\frac{1}{2}, -1)$, $(-\frac{3}{2})$, $(-\frac{1}{2}, -\frac{1}{2})$, $(-1)$, and $(-\frac{1}{2})$, satisfying $0 > i(\phi) > 1-r$ are realized by fully irreducible $\phi \in Out(F_3)$. \end{thm}

\begin{proof} For each index list we give an explicit example. We used a computer program \cite{c12} to verify that each example is indeed a train track representative of the correct rank and additionally has no pNp's. We apply the FIC to show that the example is indeed a fully irreducible outer automorphism. To verify that a given representative has PF transition matrix, since our representatives are train track maps, it suffices to prove that a sufficiently high power maps each edge over each other edge. We compute the local Whitehead graphs to show that they are connected. Having no pNp's, having PF transition matrix, and having connected local Whitehead graphs, our representatives are fully irreducible by the FIC. Since there are no pNp's, restricting to the periodic directions gives the components of the ideal Whitehead graph, from which we computed the index.

\vskip10pt

\noindent INDEX LIST $(-\frac{3}{2})$:

\vskip5pt

A plethora of examples with this index list can be found in \cite{p12d}.

\vskip10pt

\noindent INDEX LIST $(-\frac{1}{2}, -1)$:

The representative on the graph
~\\
\vspace{-9mm}
\begin{figure}[H]
\centering
\noindent \includegraphics[width=1.5in]{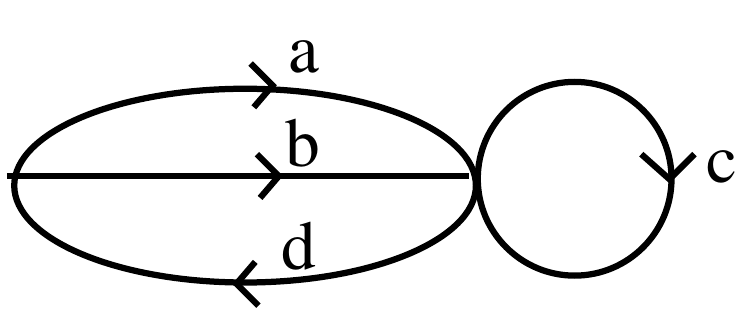}
\end{figure}

is:
~\\
\vspace{-3mm}
$$
g =
\begin{cases} a \mapsto adbcdbdbda \\
b \mapsto bdadb \\
c \mapsto cdbdbdadbdbdadbdbdbdadbdbdadbcdbdbdadbdbdadb \\
d \mapsto dbd
\end{cases}
$$

As you can see from the below figure, the local Whitehead graphs are connected.

\begin{figure}[H]
\centering
\noindent \includegraphics[width=1.7in]{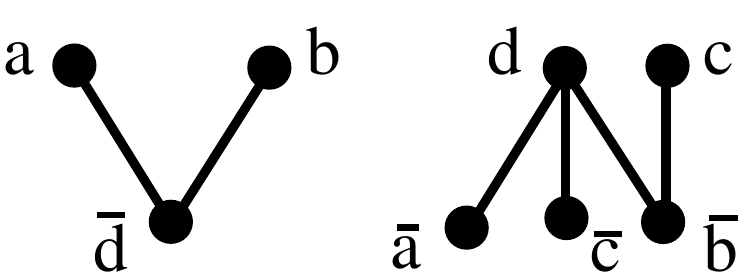}
\end{figure}

Restricting to periodic directions, since there are no periodic Nielsen paths, this gives the ideal Whitehead graph, from which the index list is computed to be $(-\frac{1}{2}, -1)$:

\begin{figure}[H]
\centering
\noindent \includegraphics[width=1.7in]{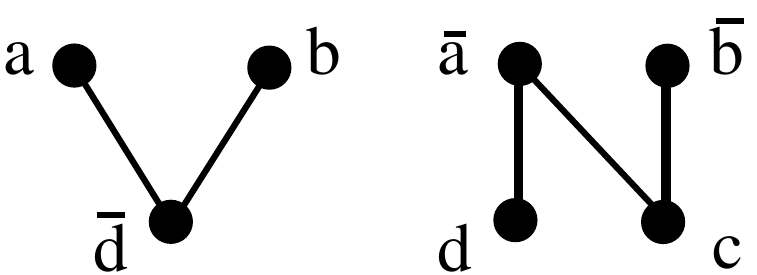}
\end{figure}

\vskip10pt

\noindent INDEX LIST $(-\frac{1}{2}, -\frac{1}{2}, -\frac{1}{2})$:

The representative on the graph
~\\
\vspace{-9mm}
\begin{figure}[H]
\centering
\noindent \includegraphics[width=1in]{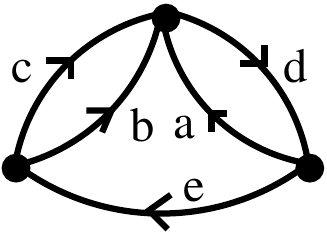}
\end{figure}

is:
~\\
\vspace{-3mm}
$$
g =
\begin{cases}
a \mapsto adecdadebda \\
b \mapsto bdadadecdadebdadadebdadebdadadecdadebdadadeb \\
c \mapsto cdadebdadadebdadebdadadebdadadecdadebdadadebdadebdadadecdadebdadadebdadebdadadecdadebda \\
d \mapsto dadebdad \\
e \mapsto ebdadade
\end{cases}
$$

As you can see from the below figure, the local Whitehead graphs are connected.

\begin{figure}[H]
\centering
\noindent \includegraphics[width=2.5in]{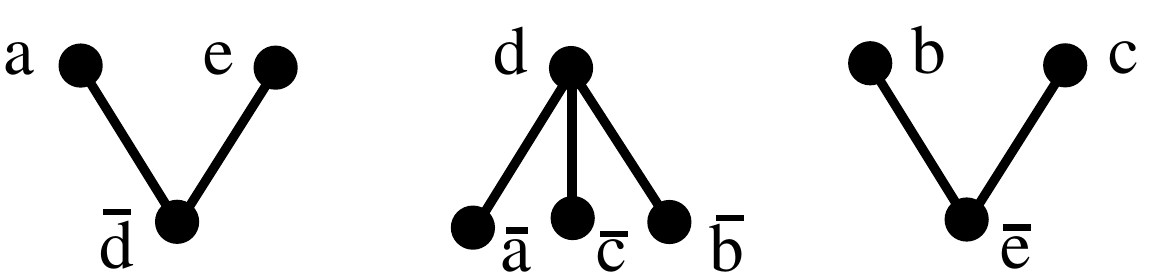}
\end{figure}

Restricting to periodic directions, since there are no periodic Nielsen paths, this gives the ideal Whitehead graph, from which the index list is computed to be $(-\frac{1}{2}, -\frac{1}{2}, -\frac{1}{2})$:

\begin{figure}[H]
\centering
\noindent \includegraphics[width=2.5in]{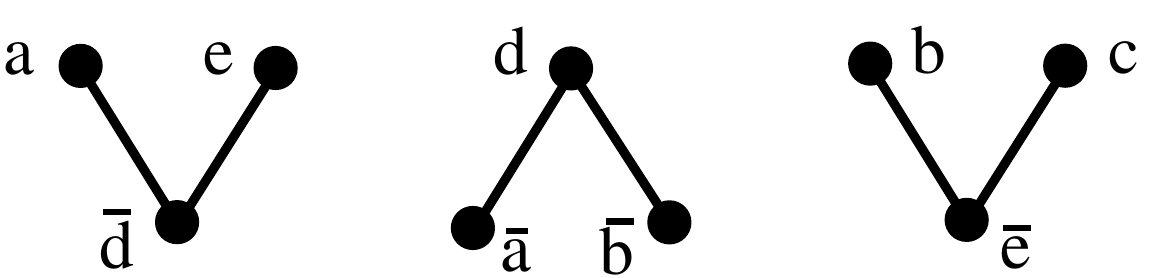}
\end{figure}

\vskip10pt

\noindent INDEX LIST $(-1)$: The representative on the rose is:
~\\
\vspace{-3mm}
$$
g =
\begin{cases} a \mapsto cab \\
b \mapsto ca \\
c \mapsto acab
\end{cases}
$$

As you can see from the below figure, the single local Whitehead graph is connected.

\begin{figure}[H]
\centering
\noindent \includegraphics[width=.85in]{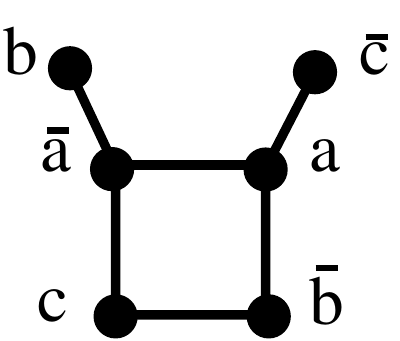}
\end{figure}

Restricting to periodic directions, since there are no periodic Nielsen paths, this gives the ideal Whitehead graph, from which the index list is computed to be $(-1)$:

\begin{figure}[H]
\centering
\noindent \includegraphics[width=.75in]{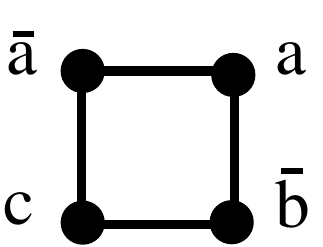}
\end{figure}

\vskip10pt

\noindent INDEX LIST $(-\frac{1}{2}, -\frac{1}{2})$:

The representative on the graph
~\\
\vspace{-9mm}
\begin{figure}[H]
\centering
\noindent \includegraphics[width=1in]{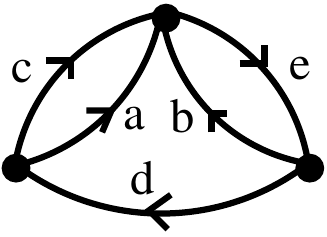}
\end{figure}

is:
~\\
\vspace{-3mm}
$$
g =
\begin{cases}
a \mapsto aebedcebedcebebedcebebeda \\
b \mapsto beda \\
c \mapsto cebebeda \\
d \mapsto dcebebed \\
e \mapsto ebedcebe
\end{cases}
$$

As you can see from the below figure, the local Whitehead graphs are connected.

\begin{figure}[H]
\centering
\noindent \includegraphics[width=2.5in]{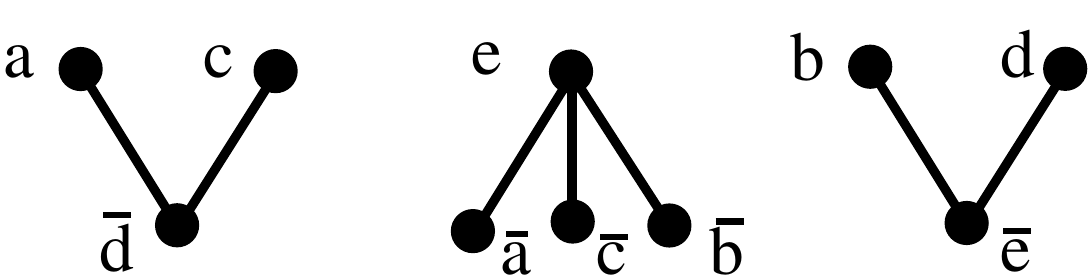}
\end{figure}

Restricting to periodic directions, since there are no periodic Nielsen paths, this gives the ideal Whitehead graph, from which the index list is computed to be $(-\frac{1}{2}, -\frac{1}{2})$:

\begin{figure}[H]
\centering
\noindent \includegraphics[width=1.25in]{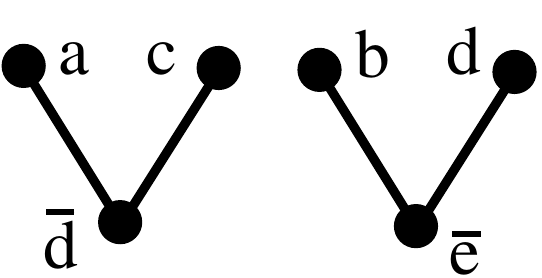}
\end{figure}

\vskip10pt

\noindent INDEX LIST $(-\frac{1}{2})$:

The representative on the graph
~\\
\vspace{-9mm}
\begin{figure}[H]
\centering
\noindent \includegraphics[width=1in]{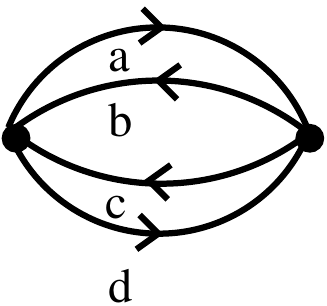}
\end{figure}

is:
~\\
\vspace{-3mm}
$$
g =
\begin{cases}
a \mapsto \bar{d}\bar{b}\bar{d}\bar{c}\bar{d}\bar{b}\bar{a}\bar{b}\bar{d}\bar{c}\bar{d}\bar{b}\bar{a}
\bar{b}\bar{d}\bar{c}\bar{d}\bar{b}\bar{d}\bar{c}\bar{d}\bar{b}
\bar{d}\bar{b}\bar{d}\bar{c}\bar{d}\bar{b}\bar{d}\bar{b}\bar{d}\bar{c}\bar{d}\bar{b}\bar{a}
\bar{b}\bar{d}\bar{c}\bar{d}\bar{b}\bar{a}\bar{b}\bar{d}\bar{c}
\bar{d}\bar{b}\bar{d}\bar{c}\bar{d}\bar{b}\bar{d}\bar{b}
\bar{d}\bar{c}\bar{d}\bar{b}\bar{d}\bar{b}\bar{d}\bar{c}\bar{d}\bar{b}\bar{a}\bar{b}\bar{d}\bar{c} \\
b \mapsto \bar{d}\bar{b}\bar{d}\bar{c}\bar{d}\bar{b}\bar{d}\bar{b}\bar{d}\bar{c}\bar{d}\bar{b}\bar{a}
\bar{b}\bar{d}\bar{c}\bar{d}\bar{b}\bar{d}
\bar{c}\bar{d}\bar{b}\bar{d}\bar{b}\bar{d}\bar{c}\bar{d}\bar{b} \\
c \mapsto cdbabdcdbdbdcdbdbdcdbdcdbabdcdbabdcdbdbdcdbdbdcdbdcdbabdcdbabdcdbd\\
\bar{a}\bar{b}\bar{d}\bar{c}\bar{d}\bar{b}\bar{a}\bar{b}\bar{d}
\bar{c}\bar{d}\bar{b}\bar{d}\bar{c}\bar{d}\bar{b}
cdbabdcdbdbdcdbdbdcdbdcdbabdcdbabdcdbdbdcdbdbdcdbdcdbabd
cdbabdcdbd\\
\bar{a}\bar{b}\bar{d}\bar{c}\bar{d}\bar{b}\bar{a}\bar{b}\bar{d}\bar{c}\bar{d}\bar{b}
\bar{d}\bar{c}\bar{d}\bar{b}\bar{d}\bar{b}\bar{d}\bar{c}\bar{d}\bar{b}\bar{d}\bar{b}\bar{d}\bar{c}
\bar{d}\bar{b}\bar{a}\bar{b}\bar{d}\bar{c}\bar{d}\bar{b}\bar{a}
\bar{b}\bar{d}\bar{c}\bar{d}\bar{b}\bar{d}\bar{c}\bar{d}\bar{b}\bar{d}\bar{b}
\bar{d}\bar{c}\bar{d}\bar{b}\bar{d}\bar{b}\bar{d}\bar{c}\bar{d}\bar{b}\bar{a}\bar{b}
\bar{d}\bar{c}\bar{d}\bar{b}\bar{d}\bar{c}\bar{d}\bar{b}\bar{d}\bar{b}\bar{d}\bar{c}\bar{d}\bar{b}
\\
d \mapsto \bar{a}\bar{b}\bar{d}\bar{c}\bar{d}\bar{b}\bar{a}\bar{b}\bar{d}\bar{c}\bar{d}\bar{b}
\bar{d}\bar{c}\bar{d}\bar{b}\bar{d}\bar{b}\bar{d}\bar{c}\bar{d}\bar{b}
\bar{d}\bar{b}\bar{d}\bar{c}\bar{d}\bar{b}\bar{a}
\bar{b}\bar{d}\bar{c}\bar{d}\bar{b}\bar{a}\bar{b}\bar{d}\bar{c}
\bar{d}\bar{b}\bar{d}\bar{c}\bar{d}\bar{b}
\end{cases}
$$

As you can see from the below figure, the local Whitehead graphs are connected.

\begin{figure}[H]
\centering
\noindent \includegraphics[width=1.75in]{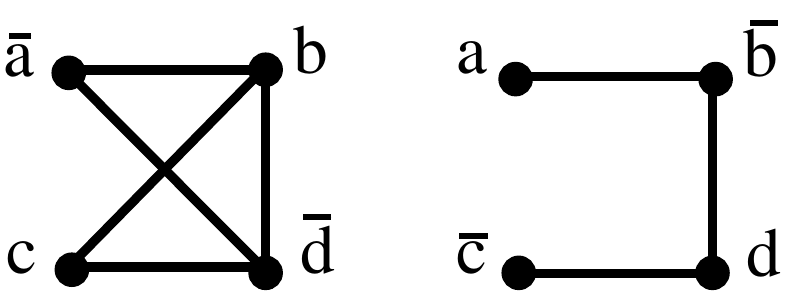}
\end{figure}

Restricting to periodic directions, since there are no periodic Nielsen paths, this gives the ideal Whitehead graph, from which the index list is computed to be $(-\frac{1}{2})$:

\begin{figure}[H]
\centering
\noindent \includegraphics[width=.75in]{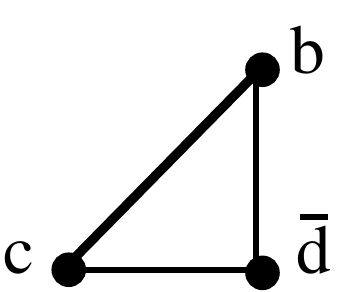}
\end{figure}

\vskip10pt

\end{proof}

\bibliographystyle{amsalpha}
\bibliography{PaperReferences}

\end{document}